\documentclass[a4paper,12pt,reqno]{amsart}
\usepackage{dsfont}
\usepackage{amsmath,amsfonts,amssymb,amsthm,enumerate,multicol}
\usepackage{tikz,graphicx}
\usepackage{hyperref}
\usepackage{caption,enumitem,wasysym}
\usepackage{cleveref}
\usepackage{epstopdf}
\usepackage{float,wrapfig}
\usepackage[labelsep=space]{caption}
\usepackage[font=footnotesize]{caption}
\usepackage{xcolor}

\newtheorem{definition}{Definition}[section]
\newtheorem{theorem}{Theorem}[section]
\newtheorem{corollary}{Corollary}[section]

\newtheorem{lemma}{Lemma}[section]

\allowdisplaybreaks

 \theoremstyle{plain}

\usepackage{mathtools}

\usepackage{amsmath,amsfonts,amssymb,amsthm,enumerate,multicol}
\usepackage{tikz}
\usepackage{float}
\usepackage{pgfplots}
\pgfplotsset{compat=1.14}
\usepackage{tkz-graph}

\usepackage{autobreak}

\allowdisplaybreaks

\numberwithin{equation}{section}

\begin{document}

\begin{center}
\Large{\textbf{Convergence of the discrete Redner–Ben-Avraham–Kahng coagulation equation}}
\end{center}
\medskip
\medskip
\centerline{${\text{Pratibha~Verma $^1$}}$}
\medskip
{\footnotesize

  \centerline{ ${}^{1}$Department of Mathematics, Indian Institute of Technology Roorkee,}
   \centerline{Roorkee-247667, Uttarakhand, India}

}

\bigskip

\section*{ABSTRACT}

This article looks at the relationship between the discrete and the continuous Redner–Ben-Avraham–Kahng (RBK) coagulation models. On the basis of a priori estimation, a weak stability principle and the weak compactness in $L_1$ for the continuous RBK model is shown. By employing a sequence of discrete models to approximate the continuous one, we show that how discrete model eventually converges to the the modified continuous one using the stability principle.

\medskip

\textbf{KeyWords :} Coagulation; Redner–Ben-Avraham–Kahng Coagulation System; Existence, Convergence.\\

\textbf{2020 Mathematics subject classification}: 45J05, 45K05, 34K30, 46B50, 26D07.

\section{Introduction}

Particles grow irreversibly in several scientific domains, including polymer chemistry, colloid science, cloud dynamics, and star formation, through successively merging clusters of particles. Coagulation is the process through which particles, molecules, or substances combine to form a larger mass or clump. A significant amount of research has previously been conducted on the kinetics of coagulation/aggregation models \cite{Bodnar,Muller, Smoluchowski}. The process of breaking or splintering bigger particles into smaller fragments is referred to as the fragmentation process. The annihilation process is a phenomenon in which particles and antiparticles collide and consequently destroy each other. This type of processes can happen physically or chemically and have been observed in a range of fields such as biology, chemistry, and materials research. 

Depending on the physical circumstances, the cluster size (either volume or number of particles) in the most fundamental coagulation models can be either a continuous non-negative real number \cite{Muller} or a discrete positive integer \cite{Smoluchowski}. Smoluchowski's coagulation model is a prototype mean field model for the coagulation process. This model along with classical coagulation-fragmentation model have attracted a lot of attention from researchers, see, \cite{ball, Barik0, Barik1, Giri1, Laurencot2, Laurencot1, Laurencot3, stewart1, stewart2}.


Let us now turn to one such model, which is a cluster-eating model. The primary mechanism of such models is binary cluster reactions, in which particles undergo collisions resulting in the formation of a single particle. However, it is observed that the size of the resultant particle decreases. The following equation provides descriptions of such types of model.
\begin{align}\label{crbk}
\frac{\partial \mathfrak{f}}{\partial t}= Q(f)(\varsigma, t),
\end{align}
where
\begin{align*}
Q(f)(\varsigma, t)=\int_0^{\infty} \mathfrak{K}(\varsigma+\varrho,\varrho) \mathfrak{f}(\varsigma+\varrho,t) \mathfrak{f}(\varrho,t) d\varrho- \int_0^{\infty} \mathfrak{K}(\varsigma,\varrho) \mathfrak{f}(\varsigma,t) \mathfrak{f}(\varrho,t) d\varrho
\end{align*}
with initial condition
\begin{align}\label{ic}
\mathfrak{f}(\varsigma,0)= \mathfrak{f}^{\mbox{in}}(\varsigma)\ge 0
\end{align}
where $ \mathfrak{f}(\varsigma, t)$ represents the particle concentration of volume $\varsigma \in \mathbb{R}_+=(0, +\infty)$ at time $t \ge 0$. The quantity $\mathfrak{K}(\varsigma,\varrho)$ depicts the coagulation rate at which particles of volume $\varsigma$ interact with particles of volume $\varrho$ to generate particles of volume $|\varsigma-\varrho|$. The coagulation rate `$\mathfrak{K}$' is assumed to be non-negative and symmetric i.e. $0 \le \mathfrak{K}(\varsigma, \varrho)= \mathfrak{K}(\varrho, \varsigma),~\forall~ (\varsigma, \varrho) \in \mathbb{R}_+^2$. Throughout this paper, we shall refer to $\mathbb{R}_+^2$ as $\mathbb{R}_+ \times \mathbb{R}_+$. The equations \eqref{crbk}-\eqref{ic} is known as the continuous Redner–Ben-Avraham–Kahng (CRBK) coagulation equation. The CRBK equation has not been much explored. Best to our knowledge, the only papers available in the literature are \cite{ankik, ankik1, Verma1}.

Redner et al. provided the first description of a finite dimensional model given in \cite{Redner1}. This model is known as the finite dimensional discrete Redner–Ben-Avraham–Kahng (RBK) coagulation model. In order to understand the dynamics of vicious civilizations, this model has been studied in \cite{Redner2}. Later, the infinite dimensional discrete version RBK (DRBK) coagulation model is discussed in various references, see \cite{dacosta1, dacosta2}. The size distribution function $f_i(t)$ of clusters of size $i$ (or $i$-clusters) at time $t$ in the discrete case follows the following system of discrete RBK coagulation equations, which will be referred to as the DRBK equations from now on.

\begin{align}\label{drbk}
\frac{d f_i}{dt}= Q_i(t), \ \ \mbox{for} \  t \in (0,+\infty)
\end{align}
where
\begin{align*}
Q_i(t) = \sum_{j=1}^{\infty} a_{i+j,j}f_{i+j} f_{j}-\sum_{j=1}^{\infty} a_{i,j}f_{i} f_{j},
\end{align*}
with initial condition
\begin{align}
f_{i}(0) = f_{i}^0, \label{dic}
\end{align}
for $i \ge 1$. Here, $ a_{i,j}=a_{j,i} \ge 0$ denotes the discrete coagulation kernel. The equations \eqref{drbk}-\eqref{dic} is known as the discrete Redner–Ben-Avraham–Kahng (RBK) equation.

Despite the fact that some authors have taken into account the relationship between the discrete and continuous models, see the survey paper \cite[p. 127]{Drake} and \cite{Aizenman, Bruno, Bouzoubaa, Ziff}, their analysis is either done at a formal level \cite{Aizenman, Drake} or their approach is limited to either a specific fragmentation model (scaling technique \cite{Ziff}) or the coagulation model (via measure-valued solutions) in \cite{Bruno}.

The primary objective of this article is to explain the relationship between the CRBK \eqref{crbk} equation and the DRBK \eqref{drbk} equation. A similar relationship was shown in \cite{DtoC1} between the classical continuous and the discrete coagulation-fragmentation equations. Furthermore, \cite{DtoC2} demonstrates the connection between Oort-Hulst-Safronov \cite{Lachowicz} and discrete Oort-Hulst-Safronov coagulation equations. Our approach is motivated by the research done in \cite{DtoC2, DtoC1}. 

For any sequence $(\phi_i)_{i \ge 1}$, by using the symmetric property of $a_{i,j}$, we get the following expression, for all $m,n \ge 1$ and $t \in (0, +\infty)$

\begin{eqnarray}\label{wdrbk1}
\frac{d}{dt}\sum_{i=m}^{n}\phi_i f_i(t)&=&-\sum_{\mathcal{T}_1(m,n)}(\phi_i-\phi_{i-j})a_{i,j}f_{i}(t) f_j(t)\nonumber\\
& & -\sum_{\mathcal{T}_2(m,n)}\phi_i a_{i,j}f_{i}(t) f_j(t)+\sum_{\mathcal{T}_3(m,n)}\phi_{i-j} a_{i,j}f_{i}(t) f_j(t),
\end{eqnarray}
where
\begin{eqnarray*}
\mathcal{T}_1(m,n)&=& \{(i,j)\in \mathbb{N}^2; m+1 \le i < n+1, 1 \le j\le i-m\},\\
\mathcal{T}_2(m,n)&=& \{(i,j)\in \mathbb{N}^2; m \le i < n+1, j\ge i-m+1\},\\
\mathcal{T}_3(m,n)&=& \{(i,j)\in \mathbb{N}^2;i \ge n+1, i-n \le j \ge i-m\}.
\end{eqnarray*}

For the proof of this formulation, we refer to \cite[Proposition 4.1]{dacosta1}. 

Since $m$ and $n$ are arbitrary natural number. Taking $m=1$ and $n \to \infty$, \eqref{wdrbk1} becomes
\begin{eqnarray}\label{wdrbk}
\frac{d}{dt}\sum_{i=1}^{\infty} \phi_i f_i(t) =\sum_{i=1}^{\infty} \sum_{j=1}^{i-1} [\phi_{i-j} - \phi_i -\phi_j] a_{i,j} f_i(t) f_j(t).
\end{eqnarray}

The contents of this article in different sections are arranged as follows.
We define several notations in the next section that will be used to introduce continuous formulas for discrete quantities. Section 3 describes some assumptions, definitions along with the main results of the article. Last section shows the convergence of weak solutions of the discrete RBK equation \eqref{drbk}-\eqref{dic} towards the weak solution of the continuous RBK equation \eqref{crbk}-\eqref{ic} for coagulation kernels $`\mathfrak{K}$' satisfying \eqref{ker}-\eqref{ker2} in space $X_{0,1}$.

\section{Some notations}
 For fix $\epsilon \in (0,1)$ and $(\varsigma, \varrho, t)\in \mathbb{R}^2_+\times (0,+\infty)$, let us introduce the following notations, depending on $\epsilon$ to write the equation \eqref{wdrbk} as the weak formulation of the modified CRBK equation, as
\begin{eqnarray}\label{d1}
\mathfrak{f}_{\epsilon}(\varrho, t) :=\sum_{i=1}^{\infty} f_i(t) \Xi_i^{\epsilon}(\varrho),
\end{eqnarray}

\begin{eqnarray}\label{d2}
\mathfrak{K}_{\epsilon}(\varsigma, \varrho) :=\sum_{i=1}^{\infty} \sum_{j=1}^{\infty} \frac{a_{i,j}}{\epsilon} \Xi_i^{\epsilon}(\varsigma)\Xi_j^{\epsilon}(\varrho),
\end{eqnarray}
and
\begin{eqnarray}\label{d3}
\Lambda_i^{\epsilon} :=[(i-1/2)\epsilon, (i+1/2)\epsilon), \ \ \ \ i\ge 1,
\end{eqnarray}
where $\Xi_i^{\epsilon} = \mathds{1}_{\Lambda_i^{\epsilon}}$ and $\mathds{1}_{\Lambda_i^{\epsilon}}$ is the indicator function defined on $\Lambda_i^{\epsilon}$ for all $ i \ge 1$. Following that, we approximate $\epsilon$-step function of $\psi \in \mathcal{D}(\mathbb{R}_+)$ and $\varsigma \in \mathbb{R}_+$ to define a sequence $(\psi_{\epsilon})$ as
\begin{eqnarray}\label{d4}
\psi_{\epsilon}(\varsigma):= \sum_{i=1}^{\infty} \psi_{i}^{\epsilon} \Xi_i^{\epsilon}(\varsigma) \ \ \ \mbox{with} \ \ \ \psi_{i}^{\epsilon}=\frac{1}{\epsilon} \int_{\Lambda_i^{\epsilon}} \psi(\varrho) d \varrho.
\end{eqnarray}
Finally, for every $\epsilon$-step measurable function $c$ from $\mathbb{R}_+$ to $\mathbb{R}$ satisfying the form
$$c(\varsigma):=\sum_{i=1}^{\infty} c_{i} \Xi_i^{\epsilon}(\varsigma), \ \ \ c_i \in \mathbb{R}.$$
The approximation of $c(\varsigma-\varrho)$ is denoted by $T_{\epsilon}(c)$ which is defined as
\begin{eqnarray}\label{d5}
T_{\epsilon}(c)(\varsigma, \varrho) := \sum_{i=1}^{\infty} \sum_{j=1}^{i-1} c_{i-j} \Xi_i^{\epsilon}(\varsigma) \Xi_j^{\epsilon}(\varrho), \ \ \ (\varsigma,\varrho) \in \mathbb{R}^2_+.
\end{eqnarray}

 By using above notations, we write an alternative form of \eqref{wdrbk} in terms of new functions $\mathfrak{f}_{\epsilon}$, $\mathfrak{K}_{\epsilon}$ and $\psi_{\epsilon}$ as
\begin{eqnarray}\label{d7}
\frac{d}{dt}\bigg( \int_0^{\infty} \mathfrak{f}_{\epsilon} \psi_{\epsilon} d\varsigma \bigg) &=& \int_0^{\infty} \int_0^{\ell_{\epsilon}(\varsigma)} [T_{\epsilon}(\psi_{\epsilon})(\varsigma, \varrho)- \psi_{\epsilon}(\varsigma)- \psi_{\epsilon}(\varrho)]\nonumber \\
& & \hspace{2cm}\mathfrak{K}_{\epsilon}(\varsigma, \varrho) \mathfrak{f}_{\epsilon}(\varsigma) \mathfrak{f}_{\epsilon}(\varrho) d\varrho d\varsigma,
\end{eqnarray}
where
\begin{eqnarray}\label{d6}
\ell_{\epsilon}(\varsigma):= \bigg( \bigg[ \frac{\varsigma}{\epsilon} +\frac{1}{2}\bigg] +\frac{1}{2}\bigg) \epsilon,
\end{eqnarray}

and $[x]$ represents the integer part of the real number $x$.

Assuming the convergence of $(\mathfrak{f}_{\epsilon})$ towards a function $\mathfrak{f}$ and if $(\mathfrak{K}_{\epsilon})$ towards some $\mathfrak{K}$, it is possible to formally pass to the limit in \eqref{d7}. The function $\mathfrak{f}$ therefore satisfies the weak formulation of the continuous RBK equation \eqref{crbk} which can be read as

 \begin{eqnarray*}
 \frac{d}{dt}\int_0^{\infty} \psi \mathfrak{f} d\varsigma=\int_0^{\infty} \int_0^{\varsigma} [\psi(\varsigma-\varrho)-\psi(\varsigma)-\psi(\varrho)] \mathfrak{K}(\varsigma,\varrho) \mathfrak{f}(\varsigma,t) \mathfrak{f}(\varrho,t) d\varrho d\varsigma,
\end{eqnarray*}
for every $t\in(0,T], T\in (0, \infty)$ and $\psi \in L^{\infty}(0,\infty)$.

\section{ASSUMPTIONS, DEFINITIONS AND MAIN RESULT}
We make the following assumptions on $\mathfrak{K}$ which is non-negative and symmetric:
\begin{eqnarray}\label{ker}
0 \le \mathfrak{K}(\varrho, \varsigma)= \mathfrak{K}(\varsigma, \varrho)=r(\varsigma) r(\varrho)+\alpha(\varsigma, \varrho), \hspace{1.5cm}(\varsigma, \varrho)\in \mathbb{R}_+^2
\end{eqnarray}
where $r$ and $\alpha$ are non-negative functions satisfying
\begin{eqnarray}\label{ker1}
\left\{
  \begin{array}{ll}
    r \in \mathcal{C}(\mathbb{R}_+; \mathbb{R}_+),\ \ \ \ \ \ \ \alpha \in \mathcal{C}(\mathbb{R}_+^2; \mathbb{R}_+), \\
    0\le \alpha(\varsigma, \varrho)= \alpha(\varrho, \varsigma)\le Ar(\varsigma)r(\varrho),\ \ \ \ \ (\varsigma, \varrho) \in [1, +\infty)^2,
  \end{array}
\right.
\end{eqnarray}
for some positive real number $A$. Additionally, we assume that $\mathfrak{K}$ is strictly subquadratic, so that for each $R\ge 1$,
\begin{eqnarray}\label{ker2}
\omega_R(\varrho)=\sup_{\varsigma \in [0,R]} \frac{\mathfrak{K}(\varsigma, \varrho)}{\varrho} \to 0 \ \ \ \mbox{as} \ \ \  \varrho \to \infty.
\end{eqnarray}
Finally, we suppose that the initial datum $\mathfrak{f}^{\text{in}}$ satisfies the condition
\begin{eqnarray}\label{ic1}
\mathfrak{f}^{\mbox{in}}\in X_{0,1}^+,
\end{eqnarray}
where $X_{0,1}^+$ is a positive cone of the Banach space
$$X_{0,1}=L^1(0, +\infty; (1+\varsigma) d\varsigma),$$
endowed with the norm $\|. \|_{X_{0,1}}$ defined by
$$\|x\|_{0,1}=\int_0^{\infty} (1+\varsigma) |x(\varsigma)| d \varsigma,\ \ \ \ \ x\in X_{0,1}.$$
Thus, $$X_{0,1}^+ := \{x\in X_{0,1}, x \ge 0 \}.$$

Next we, introduce the notion of mild and weak solutions, respectively, to the continuous RBK equation \eqref{crbk}-\eqref{ic}.

\begin{definition}\label{defmild}{\textbf{(Mild solution)}}
Let $T\in (0,\infty]$. Assume that $\mathfrak{f}^{\text{in}}$ satisfies \eqref{ic1} and the coagulation kernel `$\mathfrak{K}$' satisfies \eqref{ker}-\eqref{ker2}. A mild solution of \eqref{crbk}-\eqref{ic} on $[0,T)$ is a non-negative real-valued function $\mathfrak{f} : [0,T) \to X_{0,1}^+$ such that for every $t\in [0,T)$,
\begin{eqnarray}\label{space1}
\mathfrak{f}\in \mathcal{C}([0,T); L^1(0, \infty))\cap L^{\infty}(0, T; X_{0,1}),
\end{eqnarray}

\begin{eqnarray}\label{space2}
(\varrho,s) \mapsto \mathfrak{K}(\varsigma , \varrho) \mathfrak{f} (\varsigma ,s) \mathfrak{f} (\varrho ,s) \in L^1((0, \infty) \times (0,t)),
\end{eqnarray}
and for almost every $\varsigma \in \mathbb{R}_+$,
 \begin{eqnarray}\label{msol}
 \mathfrak{f}(\varsigma,t)&=&\mathfrak{f}^{\text{in}}(\varsigma)+\int_0^t \int_0^{\infty} \mathfrak{K}(\varsigma+\varrho,\varrho) \mathfrak{f}(\varsigma+\varrho,s) \mathfrak{f}(\varrho,s) d\varrho ds \nonumber\\
 & &- \int_0^t \int_0^{\infty} \mathfrak{K}(\varsigma,\varrho) \mathfrak{f}(\varsigma,s) \mathfrak{f}(\varrho,s) d\varrho ds .
\end{eqnarray}
\end{definition}

\begin{definition}{\textbf{(Weak solution)}}\label{defweak}
 Let $T\in (0,\infty]$, $\mathfrak{f}^{\text{in}} \in X_{0,1}^+$, and the coagulation kernel `$\mathfrak{K}$' satisfies \eqref{ker}-\eqref{ker2}. A real-valued function $\mathfrak{f}= \mathfrak{f}(\varsigma, t)$ is said to be a weak solution to the continuous RBK equation \eqref{crbk}-\eqref{ic} on $[0,T)$ if
\begin{eqnarray}\label{spacew1}
0 \le \mathfrak{f}\in \mathcal{C}([0,T); L^1(0, \infty))\cap L^{\infty}(0, T; X_{0,1}),
\end{eqnarray}

and satisfies
 \begin{eqnarray}\label{wsol}
 \int_0^{\infty} (\mathfrak{f}(\varsigma,t)-\mathfrak{f}^{\text{in}}(\varsigma))\phi(\varsigma) d\varsigma=\int_0^t\int_0^{\infty} \int_0^{\varsigma} \tilde{\phi}(\varsigma,\varrho) \mathfrak{K}(\varsigma,\varrho) \mathfrak{f}(\varsigma,s) \mathfrak{f}(\varrho,s) d\varrho d\varsigma ds
\end{eqnarray}
for every $t\in[0,T)$ and $\phi \in L^{\infty}(0,\infty)$, where
\begin{eqnarray}\label{tildephi}
\tilde{\phi}(\varsigma,\varrho)=\phi(\varsigma-\varrho)-\phi(\varsigma)-\phi(\varrho),~~~~(\varsigma,\varrho)\in (0,\infty)^2.
\end{eqnarray}
\end{definition}
At first glance, it may appear that being a mild solution in the sense of Definition \ref{defmild} is a stronger solution concept than the notion of a weak solution according to Definition \ref{defweak}.

Our main result is stated in Theorem \ref{distocont}, which shows that how the weak formulation of the continuous RBK equation can be approximated by discrete equations \eqref{wdrbk}. Fix $\epsilon \in (0,1)$ and for each $i, j \ge 1$, we define the discrete coagulation coefficients $a_{i,j}^{\epsilon}$ either by
\begin{eqnarray}\label{da1}
a_{i,j}^{\epsilon} =\frac{1}{\epsilon} \int_{\Lambda_i^{\epsilon} \times \Lambda_j^{\epsilon}} \mathfrak{K}(\varsigma, \varrho) d\varrho d\varsigma ,
\end{eqnarray}
or by
\begin{eqnarray}\label{da2}
a_{i,j}^{\epsilon} &=& \epsilon \mathfrak{K}(\epsilon i, \epsilon j),
\end{eqnarray}
where $\mathfrak{K}$ is a continuous function. Owing to \eqref{ker} and \eqref{ker2}, we get
\begin{eqnarray}
0 \le a_{i, j}^{\epsilon} &=& a_{j, i}^{\epsilon}, \ \ \ \ i, j \ge 1, \label{da3}\\
\lim_{j \to \infty} \frac{a_{i, j}^{\epsilon}}{j}& =& 0, \ \ \ \ \ \ \ i \ge 1. \label{da4}
\end{eqnarray}
Additionally, for the discrete RBK coagulation equation, we define the initial condition $f^{0, \epsilon}=(f_i^{0, \epsilon})_{i \ge 1}$ as
\begin{eqnarray}\label{da5}
f_i^{0, \epsilon}= \frac{1}{\epsilon}\int_{\Lambda_i^{\epsilon}} \mathfrak{f}^{\text{in}}(\varsigma) d\varsigma, \ \ \ \ \ i \ge 1.
\end{eqnarray}
We can easily verify from \eqref{ic1} that
\begin{eqnarray}\label{da6}
\epsilon^2 \sum_{i=1}^{\infty} i f_i^{0, \epsilon} \le 2 \int_0^{\infty} \varrho \mathfrak{f}^{\text{in}}(\varrho) d\varrho,
\end{eqnarray}
and
\begin{eqnarray}\label{da7}
\epsilon \sum_{i=1}^{\infty} f_i^{0, \epsilon} \le \int_0^{\infty} \mathfrak{f}^{\mbox{in}}(\varrho) d\varrho.
\end{eqnarray}
As established in \cite{dacosta1}, we consider $f^{\epsilon} := 
(f_i^{\epsilon})_{i \ge 1}$ as a solution to  the DRBK equation \eqref{drbk}-\eqref{dic} (in the sense of \cite[Definition 2.1]{dacosta1}) which satisfies
\begin{eqnarray}\label{dmass}
\sum_{i=1}^{\infty} i f_i^{\epsilon} (t) \le \sum_{i=1}^{\infty} i f_i^{0, {\epsilon}} , \ \ \ \ t\ge 0.
\end{eqnarray}
\medskip

The major attraction of the DRBK coagulation model is the following lemma. This lemma is a outcome of \eqref{wdrbk1} along with \eqref{dmass}.

\begin{lemma}\label{tail}
For every $R \ge 1$ and $\epsilon \in (0,1)$, let us consider an integer $m$ such that $m \epsilon < R \le (m+1) \epsilon$. Then the solution $f^{\epsilon} = (f_i^{\epsilon})_{i \ge 1}$ of \eqref{drbk}-\eqref{dic} (in the sense of \cite[Definition 2.1]{dacosta1}) with condition \eqref{dmass} satisfies
\begin{align}\label{eqtail}
\sum_{m+1}^{\infty} i f_i^{\epsilon}(t_2) \le \sum_{m+1}^{\infty} i f_i^{\epsilon}(t_1)\le \sum_{m+1}^{\infty} i f_i^{0, \epsilon} \ \ \forall \ t_2 \ge t_1 \ge 0.
\end{align}
\end{lemma}
\begin{proof}
Let $\epsilon \in (0,1)$. It follows from \eqref{wdrbk1} that
\begin{align*}
\sum_{m+1}^{\infty} i f_i^{\epsilon}(t_2) \le \sum_{m+1}^{\infty} i f_i^{\epsilon}(t_1)-\int_{t_1}^{t_2}\sum_{i=m+2}^{\infty} \sum_{j=1}^{i-(m+1)} j a_{i,j}^{\epsilon} f_i^{\epsilon}(s) f_j^{\epsilon}(s) ds.
\end{align*}

Since $a_{i,j}^{\epsilon}$ and $f_i^{\epsilon}$ are non-negative, we get
$$\sum_{m+1}^{\infty} i f_i^{\epsilon}(t_2) \le \sum_{m+1}^{\infty} i f_i^{\epsilon}(t_1) \ \ \forall \ t_2 \ge t_1 \ge 0.$$

This completes proof of Lemma \ref{tail}.

\end{proof}
\medskip

In the same way as in Section 1, we introduce continuous formulas for the discrete quantities $f^{\epsilon}$, $a_{i,j}^{\epsilon}$ and define new functions $\mathfrak{f}_{\epsilon}$ and $\mathfrak{K}_{\epsilon}$ (with $f_i^{\epsilon}$ and $a_{i,j}^{\epsilon}$ instead of $f_i$ and $a_{i,j}$ respectively) and for $(\varsigma, \varrho, t) \in \mathbb{R}_+^2 \times (0, +\infty)$, we set
\begin{eqnarray}
\mathfrak{f}_{\epsilon}(\varrho, t) &=&\sum_{i=1}^{\infty} f_i^{\epsilon}(t) \Xi_i^{\epsilon}(\varrho), \label{ca1}\\
\mathfrak{K}_{\epsilon}(\varsigma, \varrho) &=&\sum_{i=1}^{\infty} \sum_{j=1}^{\infty} \frac{a_{i,j}^{\epsilon}}{\epsilon} \Xi_i^{\epsilon}(\varsigma)\Xi_j^{\epsilon}(\varrho), \label{ca2}
\end{eqnarray}
Here, $\mathfrak{K}_{\epsilon}$ converges a.e. towards $\mathfrak{K}$ and satisfing \eqref{ker}-\eqref{ker2}.

With these assumptions, now we are in situation to state our main result:

\begin{theorem}\label{distocont}
Assume that the coagulation kernel $\mathfrak{K}$ satisfies \eqref{ker}-\eqref{ker2} and the initial condition $\mathfrak{f}^{\mbox{in}} \in X_{0,1}^+$. Let $f^{\epsilon}=(f_i^{\epsilon})_{i \ge 1}$ be a solution to the discrete RBK equation \eqref{drbk}-\eqref{dic} with the coagulation kernel $a_{i, j}^{\epsilon}$ defined by \eqref{da1} or \eqref{da2} and the initial data $f^{0, \epsilon}$ specified by \eqref{da5} such that \eqref{dmass} holds. Let us consider $\mathfrak{f}_{\epsilon}$ be a function as described in \eqref{ca1}. Then, given an initial data set $\mathfrak{f}^{\mbox{in}}$ and a subsequence $(\mathfrak{f}_{\epsilon_n})$ of $\mathfrak{f}_{\epsilon}$, there is a weak solution $\mathfrak{f}$ to the continuous RBK equation \eqref{crbk}-\eqref{ic} in the sense of Definition \ref{defweak} such that
$$ \mathfrak{f}_{\epsilon_n} \to \mathfrak{f} \ \ \ \ \mbox{in} \ \  \mathcal{C}([0,T]; w-L^1(\mathbb{R}_+)),$$
for every $T \in (0, +\infty)$.
\end{theorem}

\medskip

\section{\textbf{PROOF OF THEOREM \ref{distocont}}}
Let us start this section by considering the DRBK equation \eqref{drbk}-\eqref{dic} with coagulation coefficient $a_{i,j}^{\epsilon}$ and initial value $f^{0, \epsilon}$ defined by \eqref{da1} or \eqref{da2} and \eqref{da5}, respectively. The proof mainly relies on a weak $L^1$-compactness argument. In order to proceed with the analysis, it is necessary to obtain consistent estimates with respect to $\epsilon$ for the function $\mathfrak{f}_{\epsilon}$ as defined in \eqref{ca1}. These estimates guarantee that $(\mathfrak{f}_{\epsilon})$ is contained in a weakly compact set of $L^1$. Next, we proceed to evaluate the limit as $\epsilon \to 0$.

First we recall that
\begin{eqnarray}\label{1}
\|\mathfrak{f}^{\mbox{in}}\|_{0,1} =\int_0^{\infty} (1+\varsigma) \mathfrak{f}^{\mbox{in}}(\varsigma) d\varsigma,
\end{eqnarray}
and it follows from \eqref{ca1} that for each $t \ge 0$
\begin{eqnarray}
 \int_0^{\infty} \mathfrak{f}_{\epsilon}(\varsigma, t) d\varsigma &=& \epsilon \sum_{i=1}^{\infty} f_i^{\epsilon} (t), \label{2.1}\\
  \int_0^{\infty} \varsigma \mathfrak{f}_{\epsilon}(\varsigma, t) d\varsigma &=& \epsilon^2 \sum_{i=1}^{\infty} i f_i^{\epsilon} (t) \label{2.2}.
\end{eqnarray}

\medskip

\subsection{A priori estimates}

\begin{lemma}\label{me1}
Let $t >0$ and $\epsilon \in (0,1)$, then following inequalities hold tue:
\begin{eqnarray}\label{3}
\int_0^{\infty} \varsigma \mathfrak{f}_{\epsilon}(\varsigma, t) d\varsigma \le 2\|\mathfrak{f}^{\mbox{in}}\|_{0,1},
                                                                                                                                                                                                                                                                                                                                                                                                                                                                                                                                                                                                                                                                                                                                                                                                                                                                     \end{eqnarray} 
and 
\begin{eqnarray}\label{4}
\int_0^{\infty} \mathfrak{f}_{\epsilon}(\varsigma, t) d\varsigma \le \|\mathfrak{f}^{\mbox{in}}\|_{0,1}.
\end{eqnarray}
\end{lemma}
\begin{proof}
For proving \eqref{3}, it is a straightforward consequence of \eqref{da6}, \eqref{dmass} and \eqref{2.2},
\begin{eqnarray*}
\int_0^{\infty} \varsigma \mathfrak{f}_{\epsilon}(\varsigma, t) d\varsigma &=& \epsilon^2 \sum_{i=1}^{\infty} i f_i^{\epsilon}(t)\\
& \le & \epsilon^2 \sum_{i=1}^{\infty} i f_i^{0, \epsilon}(t)\\
& \le & 2\int_0^{\infty} \varsigma \mathfrak{f}^{\text{in}}(\varsigma, t) d\varsigma \le 2\|\mathfrak{f}^{\text{in}}\|_{0,1}.
\end{eqnarray*}

We now proceed to demonstrate \eqref{4}. By taking $\phi_i =1$ in \eqref{wdrbk}, the non-negativity of $a_{i,j}^{\epsilon}$ and $f^{\epsilon}$ implies that
\begin{eqnarray*}
\frac{d}{dt}\sum_{i=1}^{\infty} \epsilon f_i^{\epsilon}(t) =- \epsilon \sum_{i=1}^{\infty} \sum_{j=1}^{i-1}  a_{i,j}^{\epsilon} f_i^{\epsilon} f_j^{\epsilon} \le 0.
\end{eqnarray*}
By integrating both side with respect to $t$, we obtain
\begin{align*}
\sum_{i=1}^{\infty} \epsilon f_i^{\epsilon}(t) \le \sum_{i=1}^{\infty} \epsilon f_i^{0,\epsilon}.
\end{align*}
Owing to \eqref{da7} and \eqref{2.1}, we obtain for each $t \ge 0$,
\begin{align*}
\int_0^{\infty} \mathfrak{f}_{\epsilon}(\varsigma,t)=  \sum_{i=1}^{\infty} \epsilon f_i^{\epsilon}(t) 
\le \sum_{i=1}^{\infty} \epsilon f_i^{0, \epsilon} 
\le & \int_0^{\infty} \mathfrak{f}^{\text{in}}(\varsigma) d \varsigma \le \|\mathfrak{f}^{\text{in}}\|_{0,1}.
\end{align*}
This completes the proof of Lemma \ref{me1}.
\end{proof}

\medskip

Let us now state and prove following sequence of lemmas, which are required to show the uniform continuity of solutions by using refined version of the de la Vall\'{e}e Poussin theorem, see \cite[Theorem 8]{Laurencot3}.

\begin{lemma}\label{convex}
Consider a convex function $\sigma \in \mathcal{C}^1 ([0, \infty))$ that satisfies $\sigma(\varsigma)\ge 0, \forall \varsigma \in \mathbb{R}_+$ $\sigma(0)=0$, $ \sigma'(0)=1$ and $\sigma'$ is concave. Then for all $\varsigma, \varrho \ge 0$, 
\begin{eqnarray}\label{5}
\sigma(\varsigma) \le \varsigma \sigma'(\varsigma) \le 2 \sigma(\varsigma),
\end{eqnarray}
and
\begin{eqnarray}\label{6}
\varrho \sigma'(\varsigma) \le \sigma(\varrho)+\sigma(\varsigma).
\end{eqnarray}
\end{lemma}
For the proof of Lemma \ref{convex}, we refer \cite[Proposition 14]{Laurencot3}.

\begin{lemma}\label{dtcint1}
Let us assume a function $\sigma$ as defined in Lemma \ref{convex} and
\begin{eqnarray}\label{7}
\mathfrak{C}_{\sigma}=\int_0^{\infty} \sigma(\mathfrak{f}^{\mbox{in}}) (\varsigma) d\varsigma <+\infty.
\end{eqnarray}
Let $R \ge 1$. Then for each $t\in [0, T]$, $T \in (0, +\infty)$ and $\epsilon \in (0,1)$, there exists a constant $\mathcal{L}(R,T) >0$ such that
\begin{align}\label{8}
\int_0^{R} \sigma(\mathfrak{f}_{\epsilon}(\varsigma, t)) d\varsigma \le \mathcal{L}(R, T).
\end{align}
\end{lemma}
\begin{proof}
For every $R \ge 1$, there exists an integer $m \ge 1$ such that $m \epsilon < R \le (m+1) \epsilon$. We require the following constraints on $a_{i,j}$ and $r(\varsigma)$ to prove Lemma \ref{dtcint1}.

 Since, we know from \eqref{ker}, \eqref{ker1}, and \eqref{da2} that there exists a constant $\mathcal{L}_0$ depending on $R$ such that
\begin{align}\label{bound_a}
\sup_{i,j\le m+1} a_{i,j}^{\epsilon} \le \mathcal{L}_0 (R) \epsilon.
\end{align}

Following this, \eqref{ker}-\eqref{ker2} implies that for each $R\ge 1$, there exists a constant $\Omega_R$ such that
\begin{align}\label{bound_r}
\sup_{x \ge R} \bigg(\frac{r(x)}{x}\bigg)=\Omega_R \to 0 \ \ \mbox{as} \ \ R \to \infty.
\end{align}

Now, we turn to prove Lemma \ref{dtcint1}. By using \eqref{drbk}, \eqref{6} and change of order of summation, we conclude that
\begin{align}\label{int1}
\frac{d}{dt} \sum_{i=1}^{m+1} \epsilon \sigma(f_i^{\epsilon}) = & \epsilon \bigg[ \sum_{i=1}^{m+1} \sum_{j=1}^{\infty} \sigma'(f^{\epsilon}_{i}) a_{i+j,j}^{\epsilon} f_{i+j}^{\epsilon} f_j^{\epsilon}- \sum_{i=1}^{m+1} \sum_{j=1}^{\infty} \sigma'(f^{\epsilon}_{i}) a_{i,j}^{\epsilon} f_{i}^{\epsilon} f_j^{\epsilon}\bigg] \nonumber\\
 \le & \epsilon \bigg[\sum_{i=1}^{m+1} \sum_{j=1}^{i-1} \sigma'(f^{\epsilon}_{i-j}) a_{i,j}^{\epsilon} f_{i}^{\epsilon} f_j^{\epsilon}+\sum_{i=m+2}^{\infty} \sum_{j=i-(m+1)}^{i-1} \sigma'(f^{\epsilon}_{i-j}) a_{i,j}^{\epsilon} f_{i}^{\epsilon} f_j^{\epsilon}\bigg]\nonumber\\
\le & \underbrace{\epsilon \sum_{i=1}^{m+1} \sum_{j=1}^{i-1} \sigma'(f^{\epsilon}_{i-j}) a_{i,j}^{\epsilon} f_{i}^{\epsilon} f_j^{\epsilon}}_{:=J_{1, \epsilon}(R,t)}+\underbrace{\epsilon \sum_{i=m+2}^{\infty} \sum_{\substack{j=1 \\ 1\le i-j \le m+1}}^{m} \sigma'(f^{\epsilon}_{i-j}) a_{i,j}^{\epsilon} f_{i}^{\epsilon} f_j^{\epsilon}}_{:=J_{2, \epsilon}(R,t)}\nonumber\\
& + \underbrace{\epsilon\sum_{i=m+2}^{\infty} \sum_{\substack{j=m+1 \\ 1\le i-j \le m+1}}^{i-1} \sigma'(f^{\epsilon}_{i-j}) a_{i,j}^{\epsilon} f_{i}^{\epsilon} f_j^{\epsilon}}_{:=J_{3, \epsilon}(R,t)}.
\end{align}

Now, for $J_{1, \epsilon}(R,t)$, we use \eqref{2.1}, \eqref{4}, \eqref{6}, and \eqref{bound_a} to obtain
\begin{align*}
J_{1, \epsilon}(R,t) \le & \epsilon^2 \mathcal{L}_0(R) \sum_{i=1}^{m+1} \sum_{j=1}^{i-1} [\sigma(f^{\epsilon}_{i-j})+\sigma(f^{\epsilon}_{j})]f_i^{\epsilon}\\
\le & 2\epsilon^2 \mathcal{L}_0(R) \sum_{i=1}^{m+1} \sum_{j=1}^{i-1} \sigma(f^{\epsilon}_{j})f_i^{\epsilon}\\
\le & 2 \mathcal{L}_0(R) \|f^{\text{in}}\|_{0,1} \sum_{i=1}^{m+1} \epsilon \sigma(f^{\epsilon}_{i}).
\end{align*}

Next, we turn to show some bound on $J_{2, \epsilon}(R,t)$. Owing to \eqref{ker}-\eqref{ker2}, \eqref{da2}, \eqref{2.1}-\eqref{2.2} and \eqref{6} along with Lemma \ref{me1}, we get
\begin{align*}
J_{2, \epsilon}(R,t) \le & \epsilon^2 \sum_{i=m+2}^{\infty} \sum_{\substack{j=1 \\ 1\le i-j \le m+1}}^{m} [\sigma(f^{\epsilon}_{i-j})+\sigma(f^{\epsilon}_{j})] \mathfrak{K}(\epsilon i, \epsilon j) f_{i}^{\epsilon}\\

\le & \epsilon^3 \bigg(\sup_{y>R} \sup_{x\in [0,R]}\frac{\mathfrak{K}(x,y)}{y}\bigg) \sum_{i=m+2}^{\infty} \sum_{\substack{j=1 \\ 1\le i-j \le m+1}}^{m} [\sigma(f^{\epsilon}_{i-j})+\sigma(f^{\epsilon}_{j})] i f_{i}^{\epsilon}\\

\le & 2 \bigg(\sup_{y>R} \sup_{x\in [0,R]}\frac{\mathfrak{K}(x,y)}{y}\bigg) \bigg(\sum_{j=1}^{m+1} \epsilon\sigma(f^{\epsilon}_{j})\bigg) \bigg(\epsilon^2 \sum_{i=m+2}^{\infty} i f_{i}^{\epsilon}\bigg)\\

\le & 4\bigg(\sup_{y>R} \sup_{x\in [0,R]}\frac{\mathfrak{K}(x,y)}{y}\bigg)\|f^{\text{in}}\|_{0,1}\sum_{j=1}^{m+1} \epsilon\sigma(f^{\epsilon}_{j}).
\end{align*}

Again, employing \eqref{ker}-\eqref{ker2}, \eqref{6}, and \eqref{da2}, $J_{3, \epsilon}(R,t)$ can be established as
\begin{align*}
J_{3, \epsilon}(R,t) \le \epsilon^2(1+A) \sum_{i=m+2}^{\infty} \sum_{\substack{j=m+1 \\ 1\le i-j \le m+1}}^{i-1} [\sigma(1)+\sigma(f^{\epsilon}_{i-j})] r(\epsilon i) r(\epsilon j) f_{i}^{\epsilon} f_j^{\epsilon}.
\end{align*}

Following that, the estimates \eqref{2.2}-\eqref{3} and \eqref{bound_r} along with Lemma \ref{tail} imply that
\begin{align*}
J_{3, \epsilon}(R,t) \le & (1+A) \sigma(1) \Omega_R^2 \bigg(\epsilon^2\sum_{i=m+1}^{\infty} i f_{i}^{\epsilon} \bigg)^2 +(1+A)\Omega_R^2 \sum_{i=m+2}^{\infty} \sum_{\substack{j=m+1 \\ 1\le i-j \le m+1}}^{i-1} \sigma(f^{\epsilon}_{i-j}) \epsilon^4 ij f_{i}^{\epsilon} f_j^{\epsilon}\\

 \le & 4(1+A) \sigma(1) \Omega_R^2 \|f^{\text{in}}\|_{0,1}^2 + \mathcal{L}_1(R, \epsilon) \sum_{i=1}^{m+1}\sigma(f^{\epsilon}_{i}),
\end{align*}
where $\mathcal{L}_1(R)$ is a constant depends on $A$, $R$ and $f^{\text{in}}$ such that $\mathcal{L}_1(R, \epsilon) \to 0$ as $R \to \infty$.


By inserting the bounds on $J_{1, \epsilon}(R,t)$, $J_{2, \epsilon}(R,t)$, and $J_{3, \epsilon}(R,t)$ into \eqref{int1} and the integrating with respect to t, we obtain
\begin{align*}
\sum_{i=1}^{m+1} \epsilon \sigma(f_i^{\epsilon}) \le & \sum_{i=1}^{m+1} \epsilon \sigma(f_i^{0,\epsilon})+ T \mathcal{L}_2(R)+\mathcal{L}_3(R)\int_0^{T} \sum_{i=1}^{m+1} \epsilon \sigma(f_i^{\epsilon}),
\end{align*}
where $\mathcal{L}_2(R)$ and $\mathcal{L}_3(R)$ are some constants depends on $A$, $R$, $\epsilon$ and $f^{\text{in}}$, written as
\begin{eqnarray*}
	\mathcal{L}_2(R)&=& 4(1+A) \sigma(1) \Omega_R^2 \|f^{\text{in}}\|_{0,1}^2\\
	\mathcal{L}_3(R)&=& 2 \mathcal{L}_0(R) \|f^{\text{in}}\|_{0,1}+ 4\bigg(\sup_{y>R} \sup_{x\in [0,R]}\frac{\mathfrak{K}(x,y)}{y}\bigg)\|f^{\text{in}}\|_{0,1}+\mathcal{L}_1(R, \epsilon)
\end{eqnarray*}

Therefore, by using Gronwall's Inequality in above estimate, we conclude that
\begin{align*}
\sum_{i=1}^{m+1} \epsilon \sigma(f_i^{\epsilon}) \le & \bigg[\sum_{i=1}^{m+1} \epsilon \sigma(f_i^{0,\epsilon})+ T \mathcal{L}_2(R)\bigg] \exp (T\mathcal{L}_3(R)).
\end{align*}

Finally, Jensen's inequality along with \eqref{7} gives
\begin{align*}
\int_0^R \sigma(\mathfrak{f}_{\epsilon}(\varsigma , t)) d \varsigma \le & \epsilon \sum_{i=1}^{m+1}  \sigma(f_i^{\epsilon})\\
\le & [\mathfrak{C}_{\sigma}+ T \mathcal{L}_2(R)\bigg] \exp (T\mathcal{L}_3(R))=\mathcal{L}(R,T)
\end{align*}
which completes the proof of Lemma \ref{dtcint1}.
\end{proof}

\begin{lemma}\label{dtcbound1}
For each $T \in (0, +\infty)$, $R \ge 1$ and $\psi \in L^{\infty}(0,R)$,
$$t \to \int_0^{\infty} \psi(\varsigma) \mathfrak{f}_{\epsilon} (\varsigma, t) d \varsigma $$ is bounded in $W^{1, \infty}(0,T)$.
\end{lemma}
\begin{proof}
Suppose function $\psi \in L^{\infty}(0,R)$ for $R \ge 1$. Now, for $i \ge 1$, we set
\begin{align*}
\psi_{i}^{\epsilon} = \frac{1}{\epsilon} \int_{\Lambda_i^{\epsilon}} \psi(\varsigma) d\varsigma.
\end{align*}
Let us consider an integer $m$ such that $m \epsilon \le R \le (m+1)\epsilon$, then it follows from \eqref{wdrbk1}, \eqref{da2}, \eqref{2.1}-\eqref{2.2}, \eqref{bound_a} and Lemma \ref{me1} that
\begin{align*}
\bigg| \frac{d}{dt} \int_0^{R} \psi(\varsigma)\mathfrak{f}_{\epsilon}(\varsigma, t) d\varsigma \bigg| \le &  \epsilon \|\psi\|_{L^{\infty}(0,R)} \bigg[\sum_{i=1}^{m+1} \sum_{j=1}^{i-1}  a_{i,j}^{\epsilon} f_i^{\epsilon} f_j^{\epsilon} +\sum_{i=m+2}^{\infty} \sum_{j=i-(m+1)}^{i-1}  a_{i,j}^{\epsilon} f_i^{\epsilon} f_j^{\epsilon} \bigg]\\

 \le & \|\psi\|_{L^{\infty}(0,R)}\bigg[\mathcal{L}_0(R) \|f^{\text{in}}\|_{0,1}^2+ (1+A)\sum_{i=m+1}^{\infty}\sum_{j=m+1}^{\infty} r(\epsilon i) r(\epsilon j) \epsilon^2 f_i^{\epsilon} f_j{\epsilon}\bigg]\\
 
  \le & \|\psi\|_{L^{\infty}(0,R)}\bigg[\mathcal{L}_0(R) \|f^{\text{in}}\|_{0,1}^2+ (1+A)\Omega_R^2 \bigg(\epsilon^2 \sum_{i=m+1}^{\infty} if_i^{\epsilon} \bigg)^2\bigg]\\
  
  \le & \|\psi\|_{L^{\infty}(0,R)}\bigg[\mathcal{L}_0(R) \|f^{\text{in}}\|_{0,1}^2+ 4(1+A)\Omega_R^2 \|f^{\text{in}}\|_{0,1}^2\bigg].
\end{align*}

The above inequality together with Lemma \ref{me1} gives proof of Lemma \ref{dtcbound1}.
\end{proof}
\medskip

\subsection{Weak $L^1$- compactness}

\begin{lemma}\label{weakl1}
Let $T \in(0, +\infty)$ and for each $\epsilon\in (0,1)$, $\mathfrak{f}_{\epsilon}$ is the solution of corresponding ``modified" CRBK coagulation equations. Then there exists a non-negative function $\mathfrak{f} \in L^{\infty}(0,T; L^1(\mathbb{R}_+))$ and a subsequence of $(\mathfrak{f}_{\epsilon})$ (not relabeled) such that
\begin{eqnarray}\label{9}
\mathfrak{f}_{\epsilon} \to \mathfrak{f} \ \ \ \ \mbox{in} \ \ \ \ \mathcal{C}([0,T]; w-L^1(\mathbb{R}_+)).
\end{eqnarray}
\end{lemma}
\begin{proof}
Suppose $T >0$.  According to a refined version of the Arzela-Ascoli Theorem \cite[Theorem 1.3.2]{Vrabie}), for showing that the sequence $(\mathfrak{f}_{\epsilon})$ is relatively sequentially compact in $\mathcal{C}([0,T]; w-L^1(\mathbb{R}_+))$, it suffices to check that the sequence $(\mathfrak{f}_{\epsilon})$, $n\ge 1$ has the two properties listed below

\begin{enumerate}
\item[(i)] for all $t\in [0,T]$, the set $\{\mathfrak{f}_{\epsilon}(t), \epsilon\in (0,1)\}$ is weakly relatively compact in $L^1(\mathbb{R}_+)$,\\
\item[(ii)] At every $t\in [0,T]$, the family $\{\mathfrak{f}_{\epsilon}, \epsilon\in (0,1)\}$ is weakly equicontinuous in $L^1(\mathbb{R}_+)$ (see \cite[Definition 1.3.1]{Vrabie}).
\end{enumerate}

We use the Dunford-Pettis theorem to demonstrate property (i). For this, we need to prove following corollary.
\begin{corollary}\label{C1}
Let $T \in (0, \infty)$. Let us assume that the coagulation kernel satisfies $\mathfrak{K}_{\epsilon}$ satisfies \eqref{ker}-\eqref{ker2}, then the following hold true
\begin{enumerate}
\item For each $t\in [0,T]$, there exists $\mathcal{M}(T)>0$ such that 
\begin{eqnarray}\label{10}
\int_0^{\infty} (1+\varsigma) \mathfrak{f}_{\epsilon}(\varsigma,t) d\varsigma \le \mathcal{M}(T)~~~ \text{for}~~ \epsilon \in(0,1)
\end{eqnarray}
\item Given $\varepsilon>0$, there exists $R_{\varepsilon}>0$ such that for $t\in [0,T]$
\begin{eqnarray}\label{11}
\sup_{\epsilon}\bigg\{\int_{R_{\varepsilon}}^{\infty}\mathfrak{f}_{\epsilon}(\varsigma,t)d\varsigma\bigg\}\le \varepsilon.
\end{eqnarray}
\item (\textbf{Uniform Integrability}) Choose $\varepsilon>0$. Then there exists $\delta_{\varepsilon}>0$ such that for every measurable set $E \subset (0,\infty)$, we have 
$$\int_E \mathfrak{f}_{\epsilon}(\varsigma,t) d\varsigma <\varepsilon~~~~~\text{whenever}~~|E|<\delta_{\varepsilon}$$
for all $\epsilon \in (0,1)$ and $t\in [0,T]$.
\end{enumerate}
\end{corollary}
\begin{proof}
The estimate \eqref{10} can easily be followed from Lemma \ref{me1}. Next, for every $\epsilon \in (0,1)$, let $R_{\varepsilon}>0$ be such that $R_{\varepsilon}>\dfrac{\mathcal{M}(T)}{\varepsilon}$ for some fixed  $\varepsilon >0$. Then, we obtain from \eqref{3} that
\begin{eqnarray*}
\int_{R_{\varepsilon}}^{\infty} \mathfrak{f}_{\epsilon}(\varsigma,t) d\varsigma &\le & \frac{1}{R_{\varepsilon}}
\int_{R_{\varepsilon}}^{\infty} \varsigma \mathfrak{f}_{\epsilon}(\varsigma,t) d\varsigma\\
& \le& \frac{1}{R_{\varepsilon}} \mathcal{M}(T)
\end{eqnarray*}
for each $t\in [0,T]$, which implies \eqref{11}.

Since $\mathfrak{f}^{\text{in}} \in L^1(\mathbb{R}_+)$, then there exist a function $\sigma$ satisfying all assumptions of Lemma \ref{convex} along with
$$ \lim_{\varsigma \to +\infty} \frac{\sigma(\varsigma)}{\varsigma} =\infty,$$ and 
$$ \int_0^{\infty} \sigma(\mathfrak{f}^{\text{in}})(\varsigma) d\varsigma < +\infty. $$
Then, third part of Corollary \ref{C1} is a straightforward consequence of Lemma \ref{dtcint1} along with a refines version of the de la Vall\'{e}e Poussin theorem \cite[Theorem 8]{Laurencot3} and definition of uniform integrability \cite[Definition 5]{Laurencot3}. This completes the proof of Corollary \ref{C1}.
\end{proof}
Thanks to Dunford-Pettis theorem which states that if Corollary \ref{C1} holds true then for every $t \in [0,T]$, the set $\{\mathfrak{f}_{\epsilon}(t), \epsilon\in (0,1)\}$ is weakly relatively compact in $L^1(\mathbb{R}_+)$.

Next, we need to show the weak equicontinuity of the family $\{\mathfrak{f}_{\epsilon}, \epsilon\in (0,1)\}$ in $L^1(\mathbb{R}_+)$ for each $t \in [0,T]$. Let $\psi \in L^{\infty}(\mathbb{R}_+)$, then there exists a sequence of function $(\psi_n)_{n \ge 1}$ in $\mathcal{C}_c^1(\mathbb{R}_+)$ satisfying
\begin{eqnarray}\label{12}
\begin{aligned}
 \psi_n \to \psi \ \ \ \ \mbox{a.e. in } \ \ \ \mathbb{R}_+,\\
 \|\psi_n\|_{L^{\infty}} \le \|\psi\|_{L^{\infty}}.
\end{aligned}
\end{eqnarray}
Let us fix $\delta \in (0,1)$. Then third part of Corollary \ref{C1} ensure that there exist some real number $\varepsilon(\delta) >0$ such that, for each measurable subset $E \subset\mathbb{R}_+$ with $|E| \le \varepsilon(\delta)$ 
\begin{eqnarray}\label{13}
\sup_{\epsilon \in (0,1)} \sup_{t \in [0,T]} \int_E \mathfrak{f}_{\epsilon} (\varsigma, t) d\varsigma \le \delta.
\end{eqnarray} 
In addition, \eqref{12} and the Egorov's theorem imply that there exists a measurable subset $E_{\delta}$ on $[0,\frac{1}{\delta}]$ which satisfy
$$|E_{\delta}|\le \varepsilon(\delta) \ \ \ \ \mbox{and} \ \ \ \ \ \lim_{n \to \infty} \sup_{[0, \frac{1}{\delta}] \backslash E_{\delta}} |\psi_n -\psi|=0.$$
Next, for all $t_1, t_2 \in (0,T)$ and $R\in [0, \frac{1}{\delta}]$, we obtain
\begin{eqnarray}\label{14}
\bigg| \int_0^{\infty} [\mathfrak{f}_{\epsilon}(\varsigma, t_2)-\mathfrak{f}_{\epsilon}(\varsigma, t_1)] \psi(\varsigma)d \varsigma \bigg| &\le & \underbrace{\bigg| \int_0^{R} [\mathfrak{f}_{\epsilon}(\varsigma, t_2)-\mathfrak{f}_{\epsilon}(\varsigma, t_1)] \psi_{n}(\varsigma)d \varsigma \bigg|}_{:=I^n_{1, \epsilon}} \nonumber\\
& & +\underbrace{\bigg| \int_0^{R} [\mathfrak{f}_{\epsilon}(\varsigma, t_2)-\mathfrak{f}_{\epsilon}(\varsigma, t_1)] [\psi(\varsigma)-\psi_n(\varsigma)] d \varsigma \bigg|}_{:=I^n_{2, \epsilon}}\nonumber\\
& & +\underbrace{\bigg| \int_R^{\infty} [\mathfrak{f}_{\epsilon}(\varsigma, t_2)-\mathfrak{f}_{\epsilon}(\varsigma, t_1)] \psi(\varsigma)d \varsigma \bigg|}_{:=I_{3, \epsilon}}.
\end{eqnarray}
Now, let us find estimate $I^n_{1, \epsilon}$. For this, by recalling definition of $\varepsilon(\delta)$, $E_{\delta}$ and $\psi_n$, we obtain that
\begin{eqnarray*}
I^n_{1, \epsilon} = \bigg| \int_{t_1}^{t_2} \frac{d}{ds}\bigg( \int_0^R \mathfrak{f}_{\epsilon}(\varsigma, s) \psi_{n}(\varsigma) d\varsigma \bigg) ds \bigg|,
\end{eqnarray*}
then Lemma \ref{dtcbound1} implies that there exists a constant $\mathcal{L}_4(R)$ depends on $R$, $\psi$, $f^{\text{in}}$ and $A$ such that
\begin{eqnarray*}
I^n_{1, \epsilon} \le \mathcal{L}_4(R) |t_2 -t_1|,
\end{eqnarray*}
and $$\mathcal{L}_4(R)=\|\psi\|_{L^{\infty}(0,R)} \|f^{\text{in}}\|_{0,1}^2 [\mathcal{L}_0(R) + 4(1+A)\Omega_R^2 ].$$
Next, we turn our attention to $I^n_{2, \epsilon}$, which is evaluated from Lemma \ref{me1} as
\begin{eqnarray*}
I^n_{2, \epsilon} &\le & \int_{E_{\delta}} |\mathfrak{f}_{\epsilon}(\varsigma, t_2)-\mathfrak{f}_{\epsilon}(\varsigma, t_1)||\psi(\varsigma)-\psi_n(\varsigma)| d\varsigma\\
& &+ \int_{[0,R]\backslash E_{\delta}} |\mathfrak{f}_{\epsilon}(\varsigma, t_2)-\mathfrak{f}_{\epsilon}(\varsigma, t_1)||\psi(\varsigma)-\psi_n(\varsigma)| d\varsigma \\
& \le & 4\delta \|\mathfrak{f}^{\mbox{in}}\|_{0,1} \|\psi\|_{L^{\infty}}+2\|\mathfrak{f}^{\mbox{in}}\|_{0,1} \sup_{[0,R]\backslash E_{\delta}} |\psi_n -\psi|.
\end{eqnarray*}
By following the same process as above, we obtain
$$I_{3, \epsilon} \le \frac{4\|\mathfrak{f}^{\mbox{in}}\|_{0,1}}{R} \|\psi\|_{L^{\infty}}. $$
Next, by putting bounds on $I_{1, \epsilon}$, $I_{2, \epsilon}$ and $I_{3, \epsilon}$ in \eqref{14} and then pass into the limits as $n \to +\infty$ and $\delta \to 0$, respectively, we conclude that there exists a constant $\mathcal{L}_5$ depends on $R$, $\psi$, $f^{\text{in}}$ and $A$ such that
$$\sup_{\epsilon \in (0,1)} \bigg| \int_0^{\infty} [\mathfrak{f}_{\epsilon}(\varsigma, t_2) - \mathfrak{f}_{\epsilon}(\varsigma, t_1)]\psi(\varsigma) d\varsigma \bigg| \le \mathcal{L}_5 |t_2- t_1|.$$

Therefore, with a refined version of the Arzela-Ascoli Theorem, we complete the proof of Lemma \ref{weakl1}.
\end{proof}

\medskip
\subsection{Convergence}
In this section we examine that the function derived in Lemma \ref{weakl1} is a weak solution to the continuous RBK equations \eqref{crbk}-\eqref{ic}. We consider $\psi \in \mathcal{D}(\mathbb{R}_+)$ and $\psi_{\epsilon}$ as in \eqref{d4}. It was straightforward to verify that for all $t \in (0, \infty)$, $\mathfrak{f}_{\epsilon}$ satisfies
\begin{eqnarray}\label{15}
\int_0^{\infty} \mathfrak{f}_{\epsilon}(\varsigma,t) \psi_{\epsilon}(\varsigma) d\varsigma  - \int_0^{\infty} \mathfrak{f}_{\epsilon}(\varsigma,0) \psi_{\epsilon}(\varsigma) d\varsigma &=& \int_0^t \int_0^{\infty} \int_0^{\ell_{\epsilon}(\varsigma)} [T_{\epsilon}(\psi_{\epsilon})(\varsigma, \varrho)- \psi_{\epsilon}(\varsigma)- \psi_{\epsilon}(\varrho)]\nonumber\\
& & \hspace{1cm}\mathfrak{K}_{\epsilon}(\varsigma, \varrho) \mathfrak{f}_{\epsilon}(\varsigma,s) \mathfrak{f}_{\epsilon}(\varrho,s) d\varrho d\varsigma ds.
\end{eqnarray}
It is now necessary to pass the limit as $\epsilon \to 0$ in \eqref{15}. To accomplish this, we require convergence results for $\psi_{\epsilon}$ and $\mathfrak{K}_{\epsilon}$.
\begin{lemma}\label{conv1}
For every $R>0$, the sequences $(\psi_{\epsilon})$ and $(\mathfrak{K}_{\epsilon})$ defined by \eqref{d4} and \eqref{d2}, respectively, satisfy the following inequalities:
\begin{enumerate}
\item[(i)] $\psi_{\epsilon} \to \psi$ strongly in $L^{\infty}(\mathbb{R}_+)$ and
\begin{eqnarray}\label{16}
\|\psi_{\epsilon}\|_{L^{\infty}} \le \|\psi\|_{L^{\infty}}.
\end{eqnarray}
\item[(ii)] $T_{\epsilon}(\psi_{\epsilon}) \to \{(\varsigma, \varrho, t) \mapsto \psi(\varsigma-\varrho,t)\}$ strongly in $L^{\infty}(\mathbb{R}_+)$ and
\begin{eqnarray}\label{17}
\|T_{\epsilon}(\psi_{\epsilon})\|_{L^{\infty}} \le \|\psi\|_{L^{\infty}}.
\end{eqnarray}
\item[(iii)] $\mathfrak{K}_{\epsilon} \to \mathfrak{K}$ a.e. on $\mathbb{R}_+^2$ and
\begin{eqnarray}\label{18}
\|\mathfrak{K}_{\epsilon}\|_{L^{\infty}((0,R)^2)} \le \|\mathfrak{K}\|_{L^{\infty}((0,R)^2)}.
\end{eqnarray}
\end{enumerate}
\end{lemma}
\begin{proof}
Let us begin with the proof of Lemma \ref{conv1}(i). By using the definition of $\psi_{\epsilon}$ from \eqref{d4}, we obtain
\begin{eqnarray*}
|\psi_{\epsilon}(\varsigma)| &\le & \sum_{i=1}^{\infty} \frac{1}{\epsilon} \int_{\Lambda_i^{\epsilon}} |\psi(\varrho)| d\varrho \Xi_i^{\epsilon}(\varsigma)\\
& \le & \|\psi\|_{L^{\infty}} \sum_{i=1}^{\infty} \frac{1}{\epsilon} \int_{\Lambda_i^{\epsilon}}  d\varrho \Xi_i^{\epsilon}(\varsigma) \le \|\psi\|_{L^{\infty}}.
\end{eqnarray*}
By taking supremum over $\varsigma$ on both side, we conclude that \eqref{16} holds. Next, Let us take
\begin{eqnarray*}
|\psi_{\epsilon}(\varsigma)-\psi(\varsigma)|& \le &  \sum_{i=1}^{\infty} \frac{1}{\epsilon} \int_{\Lambda_i^{\epsilon}} |\psi(\varrho)-\psi(\varsigma)| d\varrho \Xi_i^{\epsilon}(\varsigma) \\
& \le & \sup_{|\varrho-\varsigma|\le \epsilon}|\psi(\varrho)-\psi(\varsigma)|.
\end{eqnarray*}
Then, by taking supremum on both side and letting $\epsilon \to 0$, we complete proof of Lemma \ref{conv1} (i).

Next, In order to prove Lemma \ref{conv1}(ii), we infer from \eqref{d5} that
\begin{eqnarray*}
|T_{\epsilon}(\psi_{\epsilon})(\varsigma, \varrho)- \psi(\varsigma-\varrho)| &\le & \sum_{i=1}^{\infty}\sum_{j=1}^{i-1} \frac{1}{\epsilon} \int_{\Lambda_{i-j}^{\epsilon}} |\psi(z)-\psi(\varsigma-\varrho)|  \Xi_i^{\epsilon}(\varsigma) \Xi_j^{\epsilon}(\varrho) dz\\
& \le & \sup_{|z-(\varsigma-\varrho)|\le \epsilon} |\psi(z)-\psi(\varsigma-\varrho)| .
\end{eqnarray*}
Therefore, after taking supremum both side and then $\epsilon \to 0$, we obtain
$$\|T_{\epsilon}(\psi_{\epsilon})(\varsigma, \varrho)-\psi(\varsigma-\varrho)\|_{L^{\infty}} \to 0 .$$
Since, $\epsilon$ is arbitrary then we complete the proof of Lemma \ref{conv1}(ii). 
Finally, by using definition of $\mathfrak{K}_{\epsilon}$ as in \eqref{ca2}, we obtain
\begin{eqnarray*}
| \mathfrak{K}_{\epsilon} (\varsigma, \varrho)| &\le & \sum_{i=1}^{\infty} \sum_{j=1}^{\infty} \frac{1}{\epsilon^2} \int_{\Lambda_i^{\epsilon} \times \Lambda_j^{\epsilon}} |\mathfrak{K} (\varsigma' ,\varrho')| d\varrho' d\varsigma' \Xi_i^{\epsilon} \Xi_j^{\epsilon}\\
& \le & \sup_{\substack{|\varsigma| \le \epsilon \\ |\varrho| \le \epsilon}} |\mathfrak{K} (\varsigma, \varrho)|
\end{eqnarray*}
In the same way, we get
\begin{eqnarray*}
|\mathfrak{K}_{\epsilon} (\varsigma, \varrho)- \mathfrak{K} (\varsigma, \varrho)| \le \sup_{\substack{|\varsigma' -\varsigma| \le \epsilon \\ |\varrho'-\varrho| \le \epsilon}} |\mathfrak{K} (\varsigma', \varrho')- \mathfrak{K}(\varsigma, \varrho)|.
\end{eqnarray*}
By taking $\epsilon \to 0$, we complete the proof of Lemma \ref{conv1}.
\end{proof}

Let us state following lemma which is required to show the convergence, see \cite[Lemma A.2]{Laurencot9} for the proof. 
\begin{lemma}\label{conv2}
Let $(\rho_n)_{n\ge 1}$ and $(\vartheta_n)_{n\ge 1}$ be two sequences in $L^1(\mathfrak{U})$ and $L^{\infty}(\mathfrak{U})$ respectively, where $\mathfrak{U}$ be an open set of $\mathbb{R}^m$ with $ m\ge1$. We assume that there exists $\varrho$ in $L^1(\mathfrak{U})$, $\vartheta$ in $L^{\infty}(U)$ and $c>0$ such that
\begin{eqnarray*}
\rho_n \rightharpoonup \rho \ \ \ \mbox{in} \ \ L^1(\mathfrak{U})\\
\|\vartheta_n\|_{L^{\infty}} \le c \ \ \ \ \mbox{and} \ \ \ \vartheta_n \to \vartheta \ \ \ \mbox{a.e. in} \ \ \mathfrak{U}.
\end{eqnarray*}
Then 
\begin{eqnarray*}
\lim_{n\to \infty}\| \rho_n(\vartheta_n -\vartheta)\|_{L^1}=0
\end{eqnarray*}
and 
\begin{eqnarray*}
\rho_n \vartheta_n \rightharpoonup \rho \vartheta \ \ \ \mbox{in} \ \ L^1(\mathfrak{U})
\end{eqnarray*}
\end{lemma}
Now, let $\psi \in \mathcal{D}(\mathbb{R}_+)$ such that for some $L \in (1, \infty)$, $\mbox{supp}(\psi) \subset [0,L-1]$, and define $\psi_{\epsilon}$ by \eqref{d4}. Let $T>0$ and $R>L$, then from Lemma \ref{weakl1},
$$ \mathfrak{f}_{\epsilon}(\varsigma, t) \mathfrak{f}_{\epsilon}(\varrho, t) \rightharpoonup \mathfrak{f}(\varsigma, t) \mathfrak{f}(\varrho, t) \ \ \ \mbox{in} \ \  L^1((0,T) \times (0,R)^2)$$

Next, from \eqref{d6}, the definition of $\ell_{\epsilon}$ shows that for a.e. $\varsigma \in \mathbb{R}_+$
$$ \mathds{1}_{[0,\ell_{\epsilon}(\varsigma)]} \xrightarrow{\epsilon \to 0} \mathds{1}_{[0, \varsigma]} .$$
Then Lemma \ref{conv1} implies that
\begin{eqnarray*}
& &\mathfrak{K}_{\epsilon}(\varsigma, \varrho)[T_{\epsilon}(\psi_{\epsilon})(\varsigma, \varrho)- \psi_{\epsilon}(\varsigma)- \psi_{\epsilon}(\varrho)] \mathds{1}_{[0,\ell_{\epsilon}(\varsigma)]}\\
& &\hspace{4cm} \xrightarrow{\epsilon \to 0} \mathfrak{K}(\varsigma, \varrho)[\psi(\varsigma- \varrho)- \psi(\varsigma)- \psi(\varrho)] \mathds{1}_{[0,\varsigma]}.
\end{eqnarray*}
Owing the bounds on $\psi_{\epsilon}$, $T_{\epsilon}(\psi_{\epsilon})$ and $\mathfrak{K}_{\epsilon}$ as stated in Lemma \ref{conv1}, we may use Lemma \ref{conv2} to derive that
\begin{eqnarray}\label{l1}
& &\int_0^T \int_0^R \int_0^R \mathfrak{K}_{\epsilon}(\varsigma, \varrho)\mathfrak{f}_{\epsilon}(\varsigma, t) \mathfrak{f}_{\epsilon}(\varrho, t) [T_{\epsilon}(\psi_{\epsilon})(\varsigma, \varrho)- \psi_{\epsilon}(\varsigma)- \psi_{\epsilon}(\varrho)] \mathds{1}_{[0,\ell_{\epsilon}(\varsigma)]}(\varrho)d\varrho d\varsigma dt \nonumber\\
& &\xrightarrow{\epsilon \to 0} \int_0^T \int_0^R \int_0^R \mathfrak{K}(\varsigma, \varrho)\mathfrak{f}(\varsigma, t) \mathfrak{f}(\varrho, t)[\psi(\varsigma- \varrho)- \psi(\varsigma)- \psi(\varrho)] \mathds{1}_{[0,\varsigma]}(\varrho) d\varrho d\varsigma dt.
\end{eqnarray}

It can easily be seen from \eqref{ker}-\eqref{ker2} and \eqref{ca2} that
\begin{eqnarray}
\sup_{\varrho \ge R} \sup_{\varsigma \in (0,L)} \frac{\mathfrak{K}_{\epsilon}(\varsigma, \varrho)}{\varrho} = & \Bar{\omega}_L(R), \ \ \ \ \ \ \ \ \ \ \lim_{R\to \infty}\Bar{\omega}_L(R)=0 \label{C1}\\
\sup_{\varsigma \ge R} \frac{r_{\epsilon}(\varsigma)}{\varsigma}= &\Bar{\Omega}_R. \ \ \ \ \ \ \ \ \ \ \lim_{R\to \infty}\Bar{\Omega}_R=0\label{C2}.
\end{eqnarray}
\newpage
Furthermore, we use \eqref{16}-\eqref{17}, \eqref{C1}-\eqref{C2} and Lemma \ref{me1} to obtain

\begin{eqnarray}
& &\bigg|\int \int_{\mathbb{R}_+^2 \backslash [0,R]^2} \mathfrak{K}_{\epsilon}(\varsigma, \varrho)\mathfrak{f}_{\epsilon}(\varsigma, t) \mathfrak{f}_{\epsilon}(\varrho, t) T_{\epsilon}(\psi_{\epsilon})(\varsigma, \varrho) \mathds{1}_{[0,\ell_{\epsilon}(\varsigma)]}(\varrho)d\varrho d\varsigma\bigg|\nonumber\\
& \le & \|\psi\|_{L^{\infty}} \bigg|2 \int_R^{\infty} \int_0^R \mathfrak{K}_{\epsilon}(\varsigma, \varrho) \mathfrak{f}_{\epsilon}(\varsigma, t) \mathfrak{f}_{\epsilon}(\varrho, t) d\varrho d\varsigma +\int_R^{\infty} \int_R^{\infty} \mathfrak{K}_{\epsilon}(\varsigma, \varrho) \mathfrak{f}_{\epsilon}(\varsigma, t) \mathfrak{f}_{\epsilon}(\varrho, t) d\varrho d\varsigma\bigg| \nonumber\\
& \le &  2\|\psi\|_{L^{\infty}} \bigg| \int_R^{\infty} \int_0^L \mathfrak{K}_{\epsilon}(\varsigma, \varrho) \mathfrak{f}_{\epsilon}(\varsigma, t) \mathfrak{f}_{\epsilon}(\varrho, t) d\varrho d\varsigma+ \int_R^{\infty} \int_L^{\infty} \mathfrak{K}_{\epsilon}(\varsigma, \varrho) \mathfrak{f}_{\epsilon}(\varsigma, t) \mathfrak{f}_{\epsilon}(\varrho, t) d\varrho d\varsigma\bigg| \nonumber\\
&\le & 2\|\psi\|_{L^{\infty}} \bigg| \int_R^{\infty} \int_0^L \mathfrak{K}_{\epsilon}(\varsigma, \varrho) \mathfrak{f}_{\epsilon}(\varsigma, t) \mathfrak{f}_{\epsilon}(\varrho, t) d\varrho d\varsigma+ (1+A) \int_R^{\infty} \int_L^{\infty} r_{\epsilon}(\varsigma) r_{\epsilon}(\varrho) \mathfrak{f}_{\epsilon}(\varsigma, t) \mathfrak{f}_{\epsilon}(\varrho, t) d\varrho d\varsigma\bigg|\nonumber \\
&\le & 4\|\mathfrak{f}^{\mbox{in}}\|_{0,1}^2\|\psi\|_{L^{\infty}} [ \Bar{\omega}_L(R)+ 2 \Bar{\Omega_R} \bar{\Omega}_L] \xrightarrow{R \to \infty} 0,\label{l2}
\end{eqnarray}

uniformly for all $\epsilon \in (0,1)$. Similarly, by using \eqref{ker}-\eqref{ker2} and \eqref{bound_r}, we obtain

\begin{eqnarray}
& &\bigg|\int \int_{\mathbb{R}_+^2 \backslash [0,R]^2} \mathfrak{K}(\varsigma, \varrho)\mathfrak{f}(\varsigma, t) \mathfrak{f}(\varrho, t) \psi(\varsigma- \varrho) \mathds{1}_{[0,\varsigma]}(\varrho)d\varrho d\varsigma\bigg| \nonumber\\
& & \hspace{2cm} \le 4\|\mathfrak{f}^{\mbox{in}}\|_{0,1}^2\|\psi\|_{L^{\infty}} [ {\omega}_L(R)+ 2 \Omega_R \Omega_L] \xrightarrow{R \to \infty} 0. \label{l3}
\end{eqnarray}

%
Finally, for Large enough $R$, we get
\begin{eqnarray}
	& &\int_0^{\infty}\int_0^{\ell_{\epsilon}(\varsigma)}[T_{\epsilon}(\psi_{\epsilon}) (\varsigma, \varrho)- \psi_{\epsilon}(\varsigma)- \psi_{\epsilon}(\varrho)] \mathfrak{K}_{\epsilon}(\varsigma, \varrho)\mathfrak{f}_{\epsilon}(\varrho,s) d\varrho d\varsigma \nonumber\\
	& =&\int_0^{R}\int_0^{\ell_{\epsilon}(\varsigma)}[T_{\epsilon} (\psi_{\epsilon}) (\varsigma, \varrho)- \psi_{\epsilon}(\varsigma)- \psi_{\epsilon}(\varrho)] \mathfrak{K}_{\epsilon}(\varsigma, \varrho)\mathfrak{f}_{\epsilon}(\varrho,s) d\varrho d\varsigma \nonumber\\
	& &+\int \int_{\mathbb{R}_+^2 \backslash [0,R]^2}[T_{\epsilon}(\psi_{\epsilon}) (\varsigma, \varrho)- \psi_{\epsilon}(\varsigma)- \psi_{\epsilon}(\varrho)] \mathfrak{K}_{\epsilon}(\varsigma, \varrho)\mathfrak{f}_{\epsilon}(\varrho,s) \mathds{1}_{[0,\ell_{\epsilon}(\varsigma)]}(\varrho) d\varrho d\varsigma. \label{l4}
\end{eqnarray}
Now, on the left hand side of \eqref{15}, we have to pass the limit as $\epsilon \to 0$. It follows from Lemma \ref{me1}, \ref{weakl1} and \ref{conv1} that
\begin{eqnarray*}
& &\bigg| \int_0^{\infty} \mathfrak{f}_{\epsilon}(\varsigma,t) \psi_{\epsilon}(\varsigma) d\varsigma - \int_0^{\infty} \mathfrak{f}(\varsigma,t) \psi(\varsigma) d\varsigma\bigg|\\
& & \hspace{1cm} \le \int_0^{\infty} |\psi_{\epsilon}(\varsigma)| |\mathfrak{f}_{\epsilon}(\varsigma, t)-\mathfrak{f}(\varsigma, t)| d\varsigma + \|\mathfrak{f}^{\mbox{in}}\|_{0,1} \|\psi_{\epsilon}(\varsigma) -\psi(\varsigma)\|_{L^{\infty}}.
\end{eqnarray*}
Since, $\psi_{\epsilon} \in L^{\infty}(0, \infty)$, then by lemma \ref{weakl1} and \eqref{16}, we obtain for each $t >0$
\begin{eqnarray}\label{19}
\int_0^{\infty}\mathfrak{f}_{\epsilon}(\varsigma,t) \psi_{\epsilon}(\varsigma) d\varsigma \xrightarrow{\epsilon \to 0} \int_0^{\infty} \mathfrak{f}(\varsigma,t) \psi(\varsigma) d\varsigma.
\end{eqnarray}
Moreover,
$$ \mathfrak{f}_{\epsilon} (\varsigma, 0) \to \mathfrak{f}^{\mbox{in}} \ \ \ \ \mbox{in} \ \ \ L^1(\mathbb{R}_+),$$
whence, by Lemma \ref{conv1},
\begin{eqnarray}\label{20}
\int_0^{\infty}\mathfrak{f}_{\epsilon}(\varsigma,0) \psi_{\epsilon}(\varsigma) d\varsigma \xrightarrow{\epsilon \to 0} \int_0^{\infty} \mathfrak{f}^{\mbox{in}}(\varsigma) \psi(\varsigma) d\varsigma.
\end{eqnarray}
Gathering \eqref{l1}-\eqref{20}, we are now able to pass the limit as $\epsilon \to 0$ in \eqref{15}, and demonstrate that the function $\mathfrak{f}$ meets the conditions outlined in \eqref{wsol} and, as a result, considered a weak solution to the continuous RBK equation \eqref{crbk}-\eqref{ic}. This completes proof of Theorem \ref{distocont}.

\medskip

\textbf{Acknowledgements:} The author extends sincere thanks to Prof. Ankik Kumar Giri, Associate Professor at Indian Institute of Technology Roorkee, Roorkee, for his support and belief in my abilities. His mentorship has not only impacted my professional development but has also instilled in me a deeper sense of confidence and determination.

\bibliography{Ref.bib}
\bibliographystyle{abbrv}

\end{document}